\documentclass[11pt,leqno]{amsart}
\usepackage[T1]{fontenc}
\usepackage{graphicx}
\usepackage{amsthm}
\usepackage{ulem}
\usepackage{amssymb, url, color, graphicx, amscd, mathrsfs}%pb-diagram} url,
\usepackage{amsmath, amssymb}
\usepackage{amsthm} %theorem
\newcommand{\R}{\mathbb{R}}
\newcommand{\C}{\mathbb{C}}
\newcommand{\prop}[1]{\mathcal{P}(#1)}
\newcommand{\norm}[1]{\lVert#1\rVert}
\newcommand{\abs}[1]{|#1|}
\newcommand{\innerproduct}[2]{\langle #1, #2 \rangle}
\newcommand{\wedgeproduct}[2]{#1 \wedge #2}

\usepackage{amsfonts,amsthm,amsmath,amssymb}
\usepackage{xcolor,paralist,titlesec,fancyhdr,etoolbox}

\newtheorem{theorem}{Theorem}[section]
\newtheorem{definition}[theorem]{Definition}

\newtheorem{lemma}[theorem]{Lemma}
\newtheorem{proposition}[theorem]{Proposition}
\newtheorem{corollary}[theorem]{Corollary}

\newtheorem{remark}{Remark}[section]

\numberwithin{equation}{section}

\titleformat{\section}
  {\normalfont\centering}{\thesection}{1em}{\MakeUppercase}

\begin{document}
\title[Refined Kato type inequalities and new vanishing theorems]{Refined Kato type inequalities and new vanishing theorems on complete K\"ahler and quaternionic K\"ahler manifolds}

%satisfying certain geometric inequalities}
\author[Dat]{Dinh Tien Dat}
\author[Dung]{Nguyen Thac Dung$^*$}
\author[Luo]{Yong Luo$^*$$\footnote{*Corrresponding authors}$}
%\author[Initial Surname]{author}
%\date{\today}
\address[Dinh Tien Dat]{Department of Mathematics - Mechanics - Informatics, Vietnam National University, University of Sciences, Hanoi, No. 334, Nguyen Trai Road, Thanh Xuan, Hanoi, Vietnam}
\email{dinhtiendat\_t66@hus.edu.vn}
\address[Nguyen Thac Dung]{Department of Mathematics - Mechanics - Informatics, Vietnam National University, University of Sciences, Hanoi, No. 334, Nguyen Trai Road, Thanh Xuan, Hanoi, Vietnam}
\email{dungmath@gmail.com or dungmath@vnu.edu.vn}
\address[Yong Luo]{Mathematical Science Research Center, Chongqing University of Technology,  No. 69, Hongguang Road, Chongqing, China}
\email{yongluo-math@cqut.edu.cn}

\begin{abstract}
Given a complete Riemannian manifold satisfying a weighted Poincar\'{e} inequality and having a bounded below Ricci curvature, various vanishing theorems for harmonic functions and harmonic 1-forms have been published. We generalized these results to $L^p$-integrable pluriharmonic functions and harmonic 1-forms  on complete K\"{a}hler and quaternionic K\"{a}hler manifolds respectively by utilizing the B\"ochner technique and several  refined Kato type inequalities. Moreover, we also prove  the vanishing property of pluriharmonic functions with finite $L^p$ energy on complete K\"{a}hler manifolds satisfying a Sobolev type inequality.
\end{abstract}
\maketitle
\bigskip

\section{Introduction}
Since the last decade of the 20th century, the weighted Poincar\'{e} inequality of numerous forms have been studied thoroughly due to many of its critical applications in analysis and mathematical physics. Consequently, various compelling structure theorems and vanishing properties of complete Riemannian manifolds satisfying the Ricci curvature lower bound and the weighted Poincar\'{e} inequality were derived (see \cite{[CCW], [ChZhou], [DLW], [DS19], [Fu], [FuY], [HL], [LW06], [Mun], [Vie], [Wang]}). Following the notation given in \cite{[LW06]}, an $n$-dimensional complete Riemannian manifold $M$ is said to satisfy a weighted Poincar\'{e} inequality with a nonnegative weight function $\rho(x)$, if the inequality
\begin{equation}\label{poincare}
    \int_M \rho(x)\phi^2(x)\,dV \leq \int_M \abs{\nabla \phi}^2\,dV
\end{equation}
holds true for every compactly supported smooth function $\phi$ on $M$. For the sake of brevity, we say that $M$ has property $\prop{\rho}$ if $M$ satisfies a weighted Poincar\'{e} inequality for some nonnegative weight function $\rho(x)$.

Denote the greatest lower bound of the spectrum of the Laplacian acting on $L^2$ functions by $\lambda_1(M)$ and suppose that it is positive, then the variational principle for $\lambda_1(M)$ asserts the validity of the following Poincar\'{e} inequality
    $$\lambda_1(M) \int_M \phi^2(x)\,dV \leq \int_M \abs{\nabla \phi}^2\,dV$$
for every $\phi \in C_{0}^{\infty}(M)$. This fact along with a specific lower bound for the Ricci curvature are the key ideas to prove the following vanishing theorem for $L^2$ integrable harmonic 1-forms on K\"{a}hler manifolds (see \cite{[Lam10]}).
\begin{theorem}[\cite{[Lam10]}]\label{lam1}
    Let $M^{2n}$ be a $2n$-dimensional complete K\"{a}hler manifold with $\lambda_1(M)>0$. Assume the Ricci curvature of $M$ satisfies
    $$Ric_M \geq -2\lambda_1(M) + \epsilon$$
    for some positive constant $\epsilon$. Then $H^1(L^2(M))=0$.
\end{theorem}
A similar vanishing result was also obtained for the class of quaternionic K\"{a}hler manifolds.
\begin{theorem}[\cite{[Lam10]}]\label{Lam12}
    Let $M^{4n}$ be a $4n$-dimensional complete quaternionic K\"{a}hler manifold. Assume that $\lambda_1(M)>0$ and that the Ricci curvature of $M$ satisfies
    $$Ric_M \geq -\frac{4}{3}\lambda_1(M) + \epsilon$$
    for some positive constant $\epsilon$. Then $H^1(L^2(M))=0$.
\end{theorem}
In this article, we introduce a   weighted Poincar\'{e} inequality assumption by only letting $M$ satisfy the property $\prop{\rho}$ for some nonnegative weight function $\rho$ and utilize it to derive the extended versions of the above vanishing theorems.
\begin{theorem} \label{1.3}
    Let $M$ be a $2n$-dimensional complete K\"{a}hler manifold having property $\prop{\rho}$ and $u$ a pluriharmonic function on $M$ with $\int_M\abs{\nabla u}^p<\infty\ (p>1)$. Assume that the Ricci curvature of $M$ satisfies
    $$Ric_M \geq -k\rho$$
    where $k < \frac{4}{p}$ if $p\geq2$ and $k<\frac{4p}{(4-p)^2}$ if $1<p<2$. Then $u$ is a constant function.
\end{theorem}
When $p=2$, by using a result of Li (\cite{[Li]}) which states that a harmonic function $u$ on a complete K\"ahler manifold with finite Dirichlet energy is pluriharmonic we obtain the following 
\begin{corollary}
   Let $M$ be a $2n$-dimensional complete K\"{a}hler manifold having property $\prop{\rho}$ and $u$ a harmonic function on $M$ with $\int_M\abs{\nabla u}^2<\infty$. Assume that the Ricci curvature of $M$ satisfies
    $$Ric_M \geq -k\rho$$
    where $k <2$. Then $u$ is a constant function.
\end{corollary} 
Let $\omega=du$, then we can see that that the above corollary goes back to Theorem \ref{lam1}. Hence, Theorem \ref{1.3} can be considered as a partial generalization of Theorem \ref{lam1}. It is interesting to see  whether Theorem \ref{1.3} holds true for harmonic functions or not.
\begin{theorem} \label{1.4}
    Let $M^{4n}$ be a complete quaternionic K\"{a}hler manifold having $$Ric_M=-4(n+2)g $$ which satisfies a polynomial volume growth condition, i.e,
        $$Vol(B_o(r))\leq Cr^m, m\geq1.$$
    for sufficiently large $r$ and some fixed real number $C$. Let $p\in \R$ such that
        $$2<p\leq\frac{2m}{m-1}
 \text{ if $m>1$, and }1<p<\infty \text{ if }m=1.$$
   Suppose that the first eigenvalue $\lambda_1$ of the Laplacian has a lower bound
    $$\lambda_1>\frac{4(n+2)p^2}{\bigl(4p-\frac{8}{3})}.$$
    %the Ricci curvature of $M$ satisfies
    %$$Ric_M \geq -k\rho$$
    %where $k < \frac{4}{p} - \frac{8}{3p^2} \text{ } (p>2)$.
    Then $H^1(L^p(M))=0$.
\end{theorem}
%%%%%%
%%%%%%%%%
We want to note that the bound of the constant  $p$ in Theorem \ref{1.4} is slightly insufficient due to some technical reasons arising from the quaternionic structure of $M$. The polynomial volume growth condition can be removed if $p=2$.
%%%%
%%%%
On the other hand, it is worth to note that due to Theorem 1.2 in \cite{[KLZ08]}, if $(M^{4n}, g)$ is a complete quaternionic K\"{a}hler manifold ($n\geq2$) then it must be Einstein, that is there is a constant $\delta$ such that
$$Ric_M(g)=4(n+2)\delta g.$$
Hence, if $\delta=-1$ our assumption on $\lambda_1$ implies
$$Ric_M=-4(n+2)>-\frac{4p-\frac{8}{3}}{p^2}\lambda_1.$$
When $p=2$, it recovers the assumption in Theorem \ref{Lam12}.
Moreover, by Corollary 2.6 in \cite{zhang}, we have the following stronger result.
\begin{proposition}% \label{1.6}
   Let $M^{4n}$ be a $4n$-dimensional complete quaternionic K\"{a}hler manifold $(n\geq2)$ with $\delta\geq0$
    then for any $p>0$, we have
    $H^1(L^p(M))=0$.
\end{proposition}

In the next of this paper, we prove vanishing theorems of pluriharmonic functions on complete K\"ahler manifolds satisfying a Sobolev type inequality. To state our result, at each point $x \in M$, let us denote $\lambda^{-}(x)=\min\{ 0,\lambda(x) \}$ where $\lambda(x)$ is the smallest eigenvalue of the Ricci curvature tensor of $M$.
\begin{theorem} \label{1.5}
    Let $M^{2n}$ be a $2n$-dimensional complete K\"{a}hler manifold satisfying the following Sobolev type inequality ($\nu>2$)
        $$\Biggl( \int_M \phi^{\frac{2\nu}{\nu-2}} \Biggl)^{\frac{\nu-2}{\nu}} \leq C_S\int_M \abs{\nabla \phi}^2,$$
    where $\phi$ is any compactly supported smooth function on $M$ and $C_S$ is a positive constant.  Suppose that $||\lambda^-||_{\nu/2} < \frac{4}{pC_S}$ if $p\geq2$ and $||\lambda^-||_{\nu/2} < \frac{4p}{(4-p)^2C_S}$ if $1<p<2$, where
        $$\norm{\lambda^-}_{\nu/2} := \Bigl(\int_M \bigl(\lambda^{-}\bigl)^{\frac{\nu}{2}}\Bigl)^{\frac{2}{\nu}}.$$
    Then every pluriharmonic function $u$ on $M$ with $\int_M\abs{\nabla u}^p<\infty(p>1)$ is a constant function.
\end{theorem}
Then by uisng Li's result, we have 
\begin{corollary}
    Let $M^{2n}$ be a $2n$-dimensional complete K\"{a}hler manifold satisfying the following Sobolev type inequality ($\nu>2$)
        $$\Biggl( \int_M \phi^{\frac{2\nu}{\nu-2}} \Biggl)^{\frac{\nu-2}{\nu}} \leq C_S\int_M \abs{\nabla \phi}^2,$$
    where $\phi$ is any compactly supported smooth function on $M$ and $C_S$ is a positive constant.  Suppose that $||\lambda^-||_{\nu/2} < \frac{2}{C_S}$, where
        $$\norm{\lambda^-}_{\nu/2} := \Bigl(\int_M \bigl(\lambda^{-}\bigl)^{\frac{\nu}{2}}\Bigl)^{\frac{2}{\nu}}.$$
    Then every harmonic function $u$ on $M$ with $\int_M\abs{\nabla u}^2<\infty$ is a constant function.
\end{corollary}
Finally, we give a proof for the following vanishing theorem on complete K\"{a}hler manifolds satisfying a Sobolev type inequality, assuming that its K\"{a}hler curvature tensor is bounded from below.
\begin{theorem} \label{1.7}
    Let $(M,g)$ be a complete $2n$-dimensional complete K\"{a}hler manifold satisfying the following  Sobolev type inequality ($\nu>2$)
    $$\Biggl( \int_M \phi^{\frac{2\nu}{\nu-2}} \Biggl)^{\frac{\nu-2}{\nu}} \leq C_S\int_M \abs{\nabla \phi}^2.$$
    Denote by $\mu_1\leq\ldots\leq\mu_{n^2}$ the eigenvalues of the K\"{a}hler curvature operator of $(M,g)$. Suppose that there exists a function $\kappa:M\to\R$ such that $\kappa\leq0$ and
    $$\frac{\mu_1(x)+\cdots+\mu_n(x)}{n+1} \geq \kappa(x)$$
    at every point $x\in M$. Then every pluriharmonic function $u$ on $M$ is constant, provided that $\abs{\nabla u}\in L^p\ (p>1)$,
    $\norm{\kappa}_{\nu/2} < \frac{4}{pC_S(n+1)}$
    if $p\geq2$ and
    $\norm{\kappa}_{\nu/2} <\frac{4p}{(4-p)^2C_S(n+1)}$
    if $1<p<2$, where
    $$\norm{\kappa}_{\nu/2} := \Bigl(\int_M \kappa^{\frac{\nu}{2}}\Bigl)^{\frac{2}{\nu}}.$$
\end{theorem}
By using Li's  result we have
\begin{corollary}
  Let $(M,g)$ be a complete $2n$-dimensional complete K\"{a}hler manifold satisfying the following  Sobolev type inequality ($\nu>2$)
    $$\Biggl( \int_M \phi^{\frac{2\nu}{\nu-2}} \Biggl)^{\frac{\nu-2}{\nu}} \leq C_S\int_M \abs{\nabla \phi}^2.$$
    Denote by $\mu_1\leq\ldots\leq\mu_{n^2}$ the eigenvalues of the K\"{a}hler curvature operator of $(M,g)$. Suppose that there exists a function $\kappa:M\to\R$ such that $\kappa\leq0$ and
    $$\frac{\mu_1(x)+\cdots+\mu_n(x)}{n+1} \geq \kappa(x)$$
    at every point $x\in M$. Then every harmonic function $u$ on $M$ is constant, provided that $\abs{\nabla u}\in L^2$ and
    $\norm{\kappa}_{\nu/2} < \frac{2}{C_S(n+1)},$
    where
    $$\norm{\kappa}_{\nu/2} := \Bigl(\int_M \kappa^{\frac{\nu}{2}}\Bigl)^{\frac{2}{\nu}}.$$
\end{corollary}
We note that the condition on K\"{a}hler curvature operator in Theorem \ref{1.7} is a new method given by Petersen and Wink \cite{PW} in developing the well-known Bochner technique. The new curvature condition gives us a way to control the Ricci term in Bochner formula effectively. This also means our vanishing result is new, although the assumption on Sobolev type inequality is well-known. Moreover, we would like to emphasize that our vanishing results also hold for $L^p$ energy with $(1<p<2)$ while most of results in literatures in this direction only work for $p\geq2$ (see e.g. Theorem 1.1 (iii) in \cite{[Xia]}).

\section{Preliminaries}

In \cite{[Lam10]}, the author gave proofs of some Kato type inequalities for $L^2$ harmonic $1$-forms on Riemannian manifolds. Specifically, he proved that for every finite $L^2$ harmonic 1-forms $\alpha$, it is true that $d*(\alpha \wedge \Omega)=0$ where $\Omega$ is a parallel form on $M$. From there, we can deduce a refined Kato type inequality for any $L^2$ harmonic 1-form $\omega$ on a K\"{a}hler manifold $M$
$$\abs{\nabla\omega}^2 \geq 2\abs{\nabla\abs{\omega}}^2.$$
When $p>1$, a similar approach cannot be applied since an arbitrary $L^p$-harmonic 1-form $\omega$ no longer satisfies $d*(\omega \wedge \Omega)=0$. Thus, in order to generalize previous results, we had to find another route. In this section, for pluriharmonic functions we can deduce the following refined kato inequality by using Theorem 4.2 in \cite{[Lam10]}.
\begin{lemma} \label{lemma:2.2}
    Let $M^{2n}$ be a  $2n$-dimensional complete K\"{a}hler manifold. If $u$ is a pluriharmonic function on $M$, then
        $$\abs{\nabla^2 u}^2 \geq 2\abs{\nabla \abs{\nabla u}}^2.$$
\end{lemma}

When it comes to quaternionic K\"{a}hler manifold, the situation gets a bit more complicated due to the quaternionic structure of $M$. In the wake of lucidity, we shall recall basic properties of quaternionic K\"{a}hler manifolds which will be necessary in the sequel.

\begin{definition}[\cite{[KLZ08]}]
    A quaternionic K\"{a}hler manifold $(M,g)$ is a Riemannian manifold with a rank 3 vector bundle $V \subset End(TM)$ satisfying
    \begin{itemize}
        \item In any coordinate neighborhood $U$ of $M$, there exists a local basis $\{ I,J,K \}$ of $V$ such that
        $$\begin{aligned}
            &I^2=J^2=K^2=-1 \\
            &IJ=-JI=K \\
            & JK=-KJ=I\\
            &KI=-IK=J
        \end{aligned}$$
        and
        $$\innerproduct{IX}{IY}=\innerproduct{JX}{JY}=\innerproduct{KX}{KY}=\innerproduct{X}{Y}$$
        for all $X,Y\in TM$.

        \item If $\phi \in \Gamma(V)$ then $\nabla_{X}\phi \in\Gamma(V)$ for all $X\in TM$.
    \end{itemize}
\end{definition}

\begin{definition}[\cite{[KLZ08]}]
    Let $(M,g)$ be a quaternionic K\"{a}hler manifold. We can define a 4-form $\Omega$ on $M$ by
    $$\Omega = \wedgeproduct{\omega_1}{\omega_1} + \wedgeproduct{\omega_2}{\omega_2} + \wedgeproduct{\omega_3}{\omega_3}$$
    where
        $$\begin{aligned}
            &\omega_1 = \innerproduct{\cdot}{I\cdot}\\
            &\omega_2 = \innerproduct{\cdot}{J\cdot}\\
            &\omega_3 = \innerproduct{\cdot}{K\cdot}.
        \end{aligned}$$
\end{definition}
In a joint work by Kong, Li and Zhou (see \cite{[KLZ08]}), the authors gave a detailed proof of the fact that the above 4-form determined by the quaternionic structure is indeed parallel, which is immensely useful for the following lemma.
\begin{lemma}
    Let $(M,g)$ be a complete quaternionic K\"{a}hler manifold and $\Omega$ be the parallel 4-form determined by the quaternionic structure. Assume that $M$ satisfies a polynomial volume growth condition, i.e,
        $$Vol(B_o(r))\leq Cr^m, m\geq1.$$
    for sufficiently large $r$ and some fixed real numbers $m, C$. Let $p>1$ satisfy
        $$2<p\leq\frac{2m}{m-1}
 \text{ if $m>1$, and }1<p<\infty \text{ if }m=1.$$
        %Here we define $\frac{2m}{m-1}=+\infty, \text{ if }m=1$.
    Then $d*(\wedgeproduct{\omega}{\Omega}) = 0$ for every $L^p$ harmonic $1$-form $\omega$.
\end{lemma}

\begin{proof}
    First, we note that
    \begin{equation}
        *d*(\wedgeproduct{\omega}{\Omega})=(-1)^{4n-1}d*(\wedgeproduct{\omega}{*\Omega}). \label{2.1}
    \end{equation}
    The proof for this identity  can be found in \cite{[KLZ08]} and \cite{[Lam10]}. Now let $\phi$ be a cut-off function such that $\phi=0$ in $B_o(r)$ (a geodesic ball of radius $r$ centered at some fixed point $o\in M$), $\phi=0$ outside $B_o(2r)$ and $\abs{\nabla \phi}\leq \frac{C_1}{r}$ for some positive constant $C_1$. Let $q\in\R$ such that
    $$\frac{1}{p}+\frac{1}{q}=1,$$
    then
    $$\frac{2m}{m+1}\leq q < 2 \ {\rm if} \ m>1,\ {\rm and} \ 1<q \ {\rm if} \ m=1.$$
    By H\"{o}lder's inequality, we have that
    $$\Biggl( \int_M\phi^q\abs{d*(\wedgeproduct{\omega}{\Omega})}^q \Biggl)^{\frac{2}{q}} \leq (Vol(B_o(2r)))^{\frac{2}{q}-1}\Biggl(\int_M \phi^2\abs{d*(\wedgeproduct{\omega}{\Omega})}^2 \Biggl).$$
    Thanks to the volume growth condition of $M$, this implies
        $$\Biggl( \int_M\phi^q\abs{d*(\wedgeproduct{\omega}{\Omega})}^q \Biggl)^{\frac{2}{q}} \leq C(2r)^{m(\frac{2}{q}-1)}\Biggl(\int_M \phi^2\abs{d*(\wedgeproduct{\omega}{\Omega})}^2 \Biggl).$$
    Now, we estimate the right hand side of the above inequality
        $$\begin{aligned}
            \int_M \phi^2\abs{d*(\wedgeproduct{\omega}{\Omega})}^2 &\leq \Biggl| \int_M \phi^2 d*(\wedgeproduct{\omega}{\Omega}) \wedge *d*(\wedgeproduct{\omega}{\Omega})\Biggl|\\
            & \leq \Biggl| \int_M \phi^2 d*(\wedgeproduct{\omega}{\Omega}) \wedge d*(\wedgeproduct{\omega}{*\Omega})\Biggl|\\
            & \leq \Biggl| \int_M d\phi^2 \wedge *(\wedgeproduct{\omega}{\Omega}) \wedge d*(\wedgeproduct{\omega}{*\Omega}) \Biggl|\\
            &\leq 2\Biggl( \int_M \abs{d\phi}^p \abs{*(\wedgeproduct{\omega}{\Omega})}^p\Biggl)^{\frac{1}{p}}\Biggl( \int_M \phi^q \abs{d*(\wedgeproduct{\omega}{*\Omega})}^q \Biggl)^{\frac{1}{q}}\\
            &= 2\Biggl( \int_M \abs{d\phi}^p \abs{*(\wedgeproduct{\omega}{\Omega})}^p\Biggl)^{\frac{1}{p}}\Biggl( \int_M \phi^q \abs{d*(\wedgeproduct{\omega}{\Omega})}^q \Biggl)^{\frac{1}{q}}
        \end{aligned}$$
    where the second and last equality follows from \eqref{2.1}, the third equality follows from integration by parts and the fact that $d^2=0$. Hence,
    \begin{equation}
        \Biggl( \int_M\phi^q\abs{d*(\wedgeproduct{\omega}{\Omega})}^q \Biggl)^{\frac{p}{q}} \leq C(2r)^{\frac{pm(2-q)}{q}}\Biggl(\int_M \abs{d\phi}^p\abs{*(\wedgeproduct{\omega}{\Omega})}^p \Biggl). \label{2.2}
    \end{equation}
    The fact that $\Omega$ is parallel implies
    $$\abs{*(\wedgeproduct{\omega}{\Omega})}\leq C_2\abs{\omega}$$
    for some constant $C_2$. Substituting the above into \eqref{2.2}, we obtain
        $$\int_{B_o(r)}\abs{d*(\wedgeproduct{\omega}{\Omega})}^q \leq C_3 r^{m(2-q)-q}\Bigl(\int_{B_o(2r)\setminus B_o(r)} \abs{\omega}^p\Bigl)^{\frac{q}{p}}.$$
    By letting $r\to+\infty$, the result follows from the assumption that $\abs{\omega}$ is $L^p$ integrable and $\frac{2m}{m+1}\leq q<2$ if $m>1$ and $1<q$ if $m=1.$
\end{proof}
A crucial consequence of the above result is that $\omega$ is quaternionic harmonic if $\omega$ is harmonic (see Lemma 3.1 in \cite{[KLZ08]}). From there, following the argument given in Theorem 4.2 in \cite{[Lam10]}, we can obtain the following refined Kato-type inequality.
\begin{corollary} \label{Corollary:2.6}
    Let $(M,g)$ be a complete quaternionic K\"{a}hler manifold and $\Omega$ be the parallel 4-form determined by the quaternionic structure. Assume that $M$ satisfies a polynomial volume growth condition, i.e,
        $$Vol(B_o(r))\leq Cr^m, m\geq1$$
    for sufficiently large $r$ and some fixed real numbers $m, C$. Let $p>1$ satisfy
        $$2<p\leq\frac{2m}{m-1}
 \text{ if $m>1$, and }1<p<\infty \text{ if }m=1.$$
    and let $\omega$ be an $L^p(p>1)$ harmonic 1-form. Then
        $$\abs{\nabla \omega}^2 \geq \frac{4}{3}\abs{\nabla \abs{\omega}}^2.$$
\end{corollary}
Now we are ready to prove our main theorems.

\section{Vanishing results with a weighted Poincar\'{e} inequality}
In this section, we derive several vanishing theorems for pluriharmonic functions having finite $L^p$ energy on K\"{a}hler manifolds and $L^p$ harmonic 1-forms on quaternionic K\"{a}hler manifolds, assuming that their Ricci curvatures are bounded from below and they satisfy a weighted Poincar\'e inequality. Throughout the section, $\rho:M\to\R$ is assumed to be a nonnegative weight function. Note that complete manifolds satisfing the weighted Poincar\'e inequality has infinite volume.

\begin{proof}[Proof of Theorem \ref{1.3}] Let $u$ be a pluriharmonic function on $M$ that satisfies $$\int_M \abs{\nabla u}^p <+\infty\ (p>1).$$ Now set $h=\abs{\nabla u}$, we shall prove that $h$ satisfies the Bochner formula of the following form
    \begin{equation}
        \Delta h \geq  \frac{\abs{\nabla h}^2}{h} -\rho h.  \label{3.1}
    \end{equation}

    Indeed, by virtue of the Kato type inequality derived in Lemma \ref{lemma:2.2} and the classical Bochner formula, we have that
    $$\begin{aligned}
        \frac{1}{2}\Delta h^2 &= \abs{\nabla^2 u}^2 + \innerproduct{\nabla\Delta u}{\nabla u} + Ric(\nabla u,\nabla u)\\
        &\geq 2\abs{\nabla h}^2  - k\rho h^2,
    \end{aligned}$$
     which implies our desired version of the Bochner formula. Let $\phi $ be a cut-off function such that $\phi= 1$ in $B_o(r)$ (the geodesic ball of radius $r$ centered at some fixed point $o\in M$), $\phi= 0$ outside $B_o(2r)$ and $\abs{\nabla \phi} \leq \frac{d}{r}$ for some positive constant $d$.

    \textbf{Case $p\geq2$.}  Multiply both side of \eqref{3.1} with $\phi^2 h^{p-2}$ and integrate over $M$, we obtain
             $$\int_{M} \phi^2 h^{p-1}\Delta h \geq \int_{M}\phi^2 h^{p-2} \abs{\nabla h}^2 -k\int_{M}\rho \phi^2 h^p.$$
     By utilizing the integration by parts formula and reshuffling the terms, it follows that
     \begin{equation} %\tag{3.2}
         2\int_M \phi h^{p-1} \innerproduct{\nabla\phi}{\nabla h} + p\int_M \phi^2 h^{p-2}\abs{\nabla h}^2 \leq k\int_M \rho\phi^2 h^p. \label{3.2}
     \end{equation}
     Applying the Cauchy inequality to the first term of the LHS of the above equation gives us
         $$-2\int_{M} \phi h^{p-1} \innerproduct{\nabla \phi}{\nabla h} \leq \epsilon \int_{M} \phi^2 h^{p-2} \abs{\nabla h}^2 + \frac{1}{\epsilon}\int_{M} h^p \abs{\nabla \phi}^2,$$
         where $\epsilon$ is a positive constant.
     Next, by using the weighted Poincar\'{e} inequality and Young's inequality consecutively, the following evaluation holds
         $$\begin{aligned}
             \int_{M}\rho\phi^2 h^p &\leq \int_M \abs{\nabla(\phi h^{\frac{p}{2}})}^2\\
             &\leq \frac{(1+\epsilon)p^2}{4}\int_{M} \phi^2 h^{p-2}\abs{\nabla h}^2 + \Bigl(1 + \frac{1}{\epsilon} \Bigl) \int_{M} h^p \abs{\nabla \phi}^2.
         \end{aligned}$$
     Substituting the above computations into \eqref{3.2} resulted in
         $$\Bigl(p - \epsilon - \frac{k(1+\epsilon)p^2}{4} \Bigl)\int_{M} \phi^2 h^{p-2} \abs{\nabla h}^2 \leq \Bigl( \frac{1}{\epsilon} + k\Bigl( 1 + \frac{1}{\epsilon} \Bigl) \Bigl)\int_{M} h^p \abs{\nabla \phi}^2.$$
     Due to the fact that $k < \frac{4}{p}$, it is true that for sufficiently small positive $\epsilon$ we have
         $$C(p,\epsilon) = \frac{4\epsilon(p-\epsilon)-k(1+\epsilon)p^2}{4d(1+k(\epsilon+1))} >0.$$
     Therefore,
         $$\begin{aligned}
             C(p,\epsilon)\int_{M} \phi^2 h^{p-2}\abs{\nabla h}^2 &\leq \int_{M}\abs{\nabla\phi}^2 h^p \\
             &\leq \frac{d^2}{r^2}\int_{M} h^p.
         \end{aligned}$$
     By letting $r \to \infty$, it follows that $h$ is constant on $M$ due to its $L^p$ integrability. Moreover, considering that $\int_M h^p <+\infty$ and that $M$ has infinite volume, we can conclude that $h=0$, which implies that $u$ is a constant function.

     \textbf{Case $1<p<2$.} Assume that $\delta,\epsilon,\eta$ are positive constants. Multiply both side of \eqref{3.1} with $\phi^2 (h+\delta)^{p-2}$ and integrate over $M$, we obtain
             $$\int_{M} \phi^2 (h+\delta)^{p-2}h\Delta h \geq \int_{M}\phi^2 (h+\delta)^{p-2} \abs{\nabla h}^2 -k\int_{M}\rho \phi^2 (h+\delta)^{p-2}h^2.$$
     By utilizing the integration by parts formula and reshuffling the terms, it follows that
     \begin{equation} %\tag{3.3}
     \begin{aligned}
         2\int_M \phi (h+\delta)^{p-2}h \innerproduct{\nabla\phi}{\nabla h} + (p-2)\int_M \phi^2 (h+\delta)^{p-3}h\abs{\nabla h}^2
        \\ +2\int_M\phi^2 (h+\delta)^{p-2} \abs{\nabla h}^2 \leq k\int_M \rho\phi^2(h+\delta)^{p-2}h^2. \label{3.3}
         \end{aligned}
     \end{equation}
    Note that
    $$2\int_M \phi (h+\delta)^{p-2}h \innerproduct{\nabla\phi}{\nabla h}\leq\epsilon\int_M\phi^2 (h+\delta)^{p-2} \abs{\nabla h}^2+\frac{1}{\epsilon}\int_M \abs{\nabla\phi}^2(h+\delta)^{p-2}h^2,$$
    and
    $$\int_M\phi^2 (h+\delta)^{p-2} \abs{\nabla h}^2\geq\int_M \phi^2 (h+\delta)^{p-3}h\abs{\nabla h}^2,$$
    we obtain
    \begin{equation} %\tag{3.4}
     \begin{aligned}
      (p-\epsilon)&\int_M\phi^2 (h+\delta)^{p-2} \abs{\nabla h}^2 \\
      &\leq\frac{1}{\epsilon}\int_M \abs{\nabla\phi}^2(h+\delta)^{p-2}h^2+k\int_M \rho\phi^2(h+\delta)^{p-2}h^2.
     \label{3.4}
         \end{aligned}
     \end{equation}
     Next, by using the weighted Poincar\'{e} inequality and Young's inequality consecutively, the following evaluation holds  $$\begin{aligned}
             &\int_M \rho\phi^2(h+\delta)^{p-2}h^2 
             \\\leq& \int_M \abs{\nabla(\phi h(h+\delta)^{\frac{p-2}{2}})}^2\\
             =& \int_M\abs{\nabla\phi h(h+\delta)^\frac{p-2}{2}+\phi\nabla h(h+\delta)^\frac{p-2}{2}+\frac{p-2}{2}\phi h(h+\delta)^\frac{p-4}{2}\nabla h}^2
             \\\leq& \left(1+\frac{1}{\epsilon}\right)\int_M \abs{\nabla\phi}^2(h+\delta)^{p-2}h^2+(1+\epsilon)(1+\eta)\int_M\phi^2\abs{\nabla h}^2(h+\delta)^{p-2}
             \\&+(1+\epsilon)\left(1+\frac{1}{\eta}\right)\frac{(p-2)^2}{4}\int_M\phi^2(h+\delta)^{p-4}h^2\abs{\nabla h}^2
             \\ \leq&\left(1+\frac{1}{\epsilon}\right)\int_M \abs{\nabla\phi}^2(h+\delta)^{p-2}h^2+(1+\epsilon)\left[1+\eta+(1+\frac{1}{\eta})\frac{(p-2)^2}{4}\right]
             \\ &\times\int_M\phi^2\abs{\nabla h}^2(h+\delta)^{p-2}. \end{aligned}$$
         Therefore we get
      $$\begin{aligned}
\left\{p-\epsilon-k(1+\epsilon)\left[1+\eta+(1+\frac{1}{\eta})\frac{(p-2)^2}{4}\right]\right\}&\int_M\phi^2\abs{\nabla h}^2(h+\delta)^{p-2}
\\& \leq\left(1+\frac{2}{\epsilon}\right)\int_M \abs{\nabla\phi}^2(h+\delta)^{p-2}h^2
\\ &\leq\left(1+\frac{2}{\epsilon}\right)\int_M \abs{\nabla\phi}^2h^p.
      \end{aligned}$$
      Then if $k<\frac{p}{1+\eta+(1+\frac{1}{\eta})\frac{(p-2)^2}{4}}, $ we can choose $\epsilon$ small enough such that  the coefficient of the LHS is positive for any $\eta>0$. Let $\eta=\frac{2-p}{2}$, we get that when $k<\frac{4p}{(4-p)^2}$ we can choose small enough positive constant $\epsilon$ such that  the coefficient of the LHS is positive. By letting $r \to \infty$, it follows that $h$ is constant on $M$ due to its $L^p$ integrability. Moreover, considering that $\int_M h^p <+\infty$ and that $M$ has infinite volume, we can conclude that $h=0$, which implies that $u$ is a cosntant function.
\end{proof}

\begin{remark}
    By letting $p=2$ in the assumption, a previous result by Lam (see Theorem 4.2 in \cite{[Lam10]}) can be obtained immediately as a special case of the above theorem.
\end{remark}
Now we turn our attention to the space of $L^p$ harmonic 1-forms on the class of quaternionic K\"{a}hler manifolds.

\begin{proof}[Proof of Theorem \ref{1.4}] Let $\omega \in H^1(L^p(M))$ be arbitrary and set $h = \abs{\omega}$. Following the arguments given in the proof of Theorem \ref{1.3}, which utilize the original Bochner formula and Corollay \ref{Corollary:2.6}, it can be shown that
    $$\begin{aligned}
    \Delta h &\geq \frac{1}{3}\frac{\abs{\nabla h}^2}{h} -4(n+2) h\\
    &=\frac{1}{3}\frac{\abs{\nabla h}^2}{h} -\left(\frac{4(n+2)}{\lambda_1}\right)\lambda_1 h.
    \end{aligned}$$
    Let $k:=\frac{4(n+2)}{\lambda_1}$ and note that if $\lambda_1>0$ then by the variation principle for $\lambda_1$, we have the following Poincar\'{e} inequality
    $$\lambda_1\int_M\phi^2\leq\int_M|\nabla\phi|^2$$
    for any $\phi\in C^\infty_0(M)$.
    Applying similar argument as in the proof of Theorem \ref{1.3}, one can derive
        $$\Bigl( p - \frac{2}{3} -\frac{kp^2(1+\epsilon)}{4}-\epsilon\Bigl) \int_{M} \phi^2 h^{p-2} \abs{\nabla h}^2 \leq \Bigl( k\Bigl(1+\frac{1}{\epsilon}\Bigl) - \frac{1}{\epsilon}\Bigl) \int_{M} h^p \abs{\nabla \phi}^2$$
    where $\phi \in C_{0}^{\infty}(M)$ is a cut-off function such that $\phi= 1$ in $B_o(r)$ (a geodesic ball centered at some fixed point $o$ or radius $r$), $\phi = 0$ outside $B_o(2r)$ and $\abs{\nabla \phi} \leq \frac{d}{r}$ for some positve constant $d$. Due to the fact that $k=\frac{4(n+2)}{\lambda_1}<\bigl(4p-\frac{8}{3})/p^2$, the coefficients of the terms on both sides are positive for sufficiently small $\epsilon>0$. \\
    Now set
    $$C_Q(p,\epsilon) = \frac{p-\frac{2}{3}-k(1+\epsilon)\frac{p^2}{4}-\epsilon}{d\Bigl[ \frac{4p-8/3}{p^2}\Bigl(1+\frac{1}{\epsilon} \Bigl) -\frac{1}{\epsilon}\Bigl]}.$$
    The above inequality implies that
    $$C_Q(p,\epsilon)\int_{M} \phi^2 h^{p-2} \abs{\nabla h}^2 \leq \frac{d^2}{r^2}\int_{M}h^p.$$
    Since $\omega \in H^1(L^p(M))$, we have $\int_M h^p < +\infty$. Therefore, by letting $r\to+\infty$, the integral $\frac{1}{r^2}\int_{M} h^p$ vanishes. Thus, $h^{p-2}\abs{\nabla h}^2=0$, which implies that $h$ is constant on $M$. Now suppose that $h \neq 0$, we have $\int_M h^p = \infty$, since $M$ has infinite volume, which is a contradiction. Therefore, $h \equiv 0$ and the triviality of $H^1(L^p(M))$ is justified.
\end{proof}
\begin{remark}
    Notice that when $M^{4n}$ is a quaternionic K\"{a}hler manifold, it is a well-known fact that every harmonic form $\omega$ on $M$ satisfies the following Kato type inequality
       \begin{equation}\label{e3.3}
        \abs{\nabla \omega}^2 \geq \frac{4n+1}{4n} \abs{ \nabla \abs{\omega}}^2.      \end{equation}
    By following similar arguments given in the above results, we can also prove similar vanishing properties without the limitation of $p$. However, the Kato constant of \eqref{e3.3} depends on the dimension of $M$, and is less than $4/3$. Therefore, the lower bound of $\lambda_1$ now is also depend on the Kato constant and is weaker than the bound stated in Theorem \ref{1.4}. The detailed proof for this fact is left for interested readers.
\end{remark}

\section{Vanishing results  with a Sobolev type inequality}

In this section, we take a look at complete K\"ahler  manifolds satisfying a  Sobolev type inequality. The first vanishing theorem regarding with the lower bound of the Ricci curvature of K\"{a}hler manifolds, while the last theorem is a bit more involved as it requires a more general analogue of curvature on  K\"{a}hler manifold. Thus, to maintain clearness, we recall the needed definition and notation of algebraic curvature tensor on Euclidean vector spaces and later, on K\"{a}hler manifolds, in which case it is called the K\"{a}hler curvature operator. Our choice for notation is followed by that given in \cite{[DungCho23]} and \cite{[PW22]}.

Throughout the section, let us denote the smallest eigenvalue of the Ricci curvature tensor of $M$ at each point $x \in M$ by $\lambda(x)$ and $\lambda^{-}(x)=\min\{ 0,\lambda(x) \}$. Note that complete manifolds satisfying the Sobolev type inequality in our main theorems have infinite volume.

\begin{proof}[Proof of Theorem \ref{1.5}]
First, we shall prove that $u$ satisfies the B\"ochner formula of the following form
    \begin{equation}
        \Delta\abs{\nabla u} \geq \frac{\abs{\nabla\abs{u}}^2}{\abs{\nabla u}} + \abs{\nabla u}\lambda^{-}.\label{4.1}
    \end{equation}
    Indeed, by virtue of the refined Kato type inequality that we proved in Lemma \ref{lemma:2.2} and the classical B\"ochner formula, we have
    $$
        \begin{aligned}
            \frac{1}{2}\Delta\abs{\nabla u}^2 &= \abs{\nabla^2 u}^2 + \innerproduct{\nabla\Delta u}{\nabla u} + Ric(\nabla u,\nabla u)\\
            &\geq 2 \abs{\nabla\abs{\nabla u}}^2 + \abs{\nabla u}^2\lambda \\
            &\geq 2 \abs{\nabla\abs{\nabla u}}^2 + \abs{\nabla u}^2\lambda^{-}.
        \end{aligned}
    $$
    which implies our desired version of the Bochner formula.

    For simplicity, we denote by $h=\abs{\nabla u}$ in the following.
    Now let $\phi$ be a cut-off function on $M$ such that $\phi=1$ in $B_o(r)$ (the geodesic ball of radius $r$ centered at some fixed point $o\in M$), $\phi=0$ outside $B_o(2r)$ and $\abs{\nabla\phi} < \frac{d}{r}$ for some positive constant $d$.

    \textbf{Case $p\geq2$.} Multiply both side of \eqref{4.1} with $\phi^2h^{p-2}$ and integrate over $M$, we obtain
        $$\int_M \phi^2 h^{p-1} \Delta h \geq \int_M \phi^2 h^{p-2}\abs{\nabla h}^2 + \int_M \phi^2h^p\lambda^{-}.$$
    By the integration by parts formula and reshuffling the terms, it follows that
    \begin{equation}
        2\int_M \phi h^{p-1}\innerproduct{\nabla\phi}{\nabla h} + p\int_M \phi^2 h^{p-2} \abs{\nabla h}^2 \leq \int_M (-\lambda^{-})\phi^2 h^p.\label{4.2'}
    \end{equation}
    Applying the Cauchy inequality to the LHS of the above equation gives us
    $$-2\int_M \phi h^{p-1}\innerproduct{\nabla\phi}{\nabla h} \leq \epsilon\int_M \phi^2 h^{p-2} \abs{\nabla h}^2 + \frac{1}{\epsilon} \int_M h^p \abs{\nabla\phi}^2,$$
    where $\epsilon$ is a positive constant. Next, by using H\"{o}lder inequality, the weighted Sobolev inequality and Young's inequality consecutively on the RHS term, the following evaluation holds
    $$
        \begin{aligned}
            \int_M (-\lambda^{-})\phi^2 h^p &\leq \Biggl(\int_M \bigl( -\lambda^{-}\bigl)^\frac{\nu}{2}\Biggl)^\frac{2}{\nu} \Biggl(\int_M \Bigl(\phi h^{\frac{p}{2}}\Bigl)^{\frac{2\nu}{\nu-2}}\Biggl)^{\frac{\nu-2}{\nu}} \\
            & \leq C_S||\lambda^{-}||_{\nu/2} \int_M \Big| \nabla \Bigl(\phi h^{\frac{p}{2}}\Bigl) \Big|^2 \\
            & \leq C_S||\lambda^{-}||_{\nu/2} \Biggr[ \frac{p^2(1+\epsilon)}{4}\int_M \phi^2 h^{p-2}\abs{\nabla h}^2 + \Bigl(1+\frac{1}{\epsilon}\Bigl)\int_M h^p \abs{\nabla\phi}^2 \Biggr].
        \end{aligned}
    $$
    Substituting the above computations into \eqref{4.2'}, we obtain
    $$
        \Biggr[ p-\epsilon -\frac{C_S p^2(1+\epsilon)}{4}||\lambda^{-}||_{\nu/2} \Biggr] \int_M \phi^2 h^{p-2}\abs{\nabla h}^2 \leq \Biggr[ C_S\Bigl(1+\frac{1}{\epsilon}\Bigl)||\lambda^{-}||_{\nu/2} +\frac{1}{\epsilon} \Biggr]\int_M h^p \abs{\nabla\phi}^2.
    $$
    Due to the fact that $\norm{\lambda^-}_{\nu/2} < \frac{4}{pC_S}$, we can choose $\epsilon$ be a small enough potitive constant such that
        $$D(n,\epsilon) = \frac{p+\epsilon -\frac{C_S p^2(1+\epsilon)}{4}||\lambda^{-}||_{\nu/2}}{C_S\Bigl(1+\frac{1}{\epsilon}\Bigl)||\lambda^{-}||_{\nu/2} +\frac{1}{\epsilon}} >0 \quad .$$
    Hence
    $$D(n,\epsilon) \int_{B_o(r)} h^{p-2} \abs{\nabla h}^2 \leq \frac{d^2}{r^2}\int_{B_o(r)}h^p.$$
    By letting $r\to+\infty$, it holds true that $h=0$ on $M$ thanks to the finiteness of the $L^p$- energy of $h$ and that $M$ has infinite volume. Therefore $u$ is a contant function on $M$.

    \textbf{Case $1<p<2$.} Assume $\delta, \eta$ are positive constant in the following. Multiply both side of \eqref{4.1} with $\phi^2(h+\delta)^{p-2}$ and integrate over $M$, we obtain
        $$\int_M \phi^2 (h+\delta)^{p-2} h\Delta h \geq \int_M \phi^2 (h+\delta)^{p-2}\abs{\nabla h}^2 + \int_M \phi^2(h+\delta)^{p-2}h^2\lambda^{-}.$$
        By the integration by parts formula and reshuffling the terms, it follows that
    \begin{equation}
    \begin{aligned}
         &2\int_M \phi (h+\delta)^{p-2} h\innerproduct{\nabla\phi}{\nabla h} + (p-2)\int_M \phi^2 (h+\delta)^{p-3}h \abs{\nabla h}^2\\
         &+2\int_M\phi^2 (h+\delta)^{p-2}\abs{\nabla h}^2
         \leq \int_M (-\lambda^{-})\phi^2 (h+\delta)^{p-2}h^2.\label{4.3}
    \end{aligned}
    \end{equation}
  Note that
  $$\int_M\phi^2 (h+\delta)^{p-2}\abs{\nabla h}^2\geq\int_M \phi^2 (h+\delta)^{p-3}h \abs{\nabla h}^2,$$
and applying the Cauchy inequality  gives us
    $$-2\int_M \phi (h+\delta)^{p-2}h\innerproduct{\nabla\phi}{\nabla h} \leq \epsilon\int_M \phi^2 (h+\delta)^{p-2} \abs{\nabla h}^2 + \frac{1}{\epsilon} \int_M (h+\delta)^{p-2}h^2 \abs{\nabla\phi}^2,$$
    where $\epsilon$ is a positive constant. Next, by using H\"{o}lder inequality, the weighted Sobolev inequality and Young's inequality consecutively, the following evaluation holds
    $$
        \begin{aligned}
           & \int_M (-\lambda^{-})\phi^2 (h+\delta)^{p-2}h^2 \leq \Biggl(\int_M \bigl( -\lambda^{-}\bigl)^\frac{\nu}{2}\Biggl)^\frac{2}{\nu} \Biggl(\int_M \Bigl(\phi h(h+\delta)^{\frac{p-2}{2}}\Bigl)^{\frac{2\nu}{\nu-2}}\Biggl)^{\frac{\nu-2}{\nu}} \\
             \leq& C_S||\lambda^{-}||_{\nu/2} \int_M \Big| \nabla \Bigl(\phi h(h+\delta)^{\frac{p-2}{2}}\Bigl) \Big|^2
        \\=&C_S||\lambda^{-}||_{\nu/2}
            \int_M\abs{\nabla\phi h(h+\delta)^\frac{p-2}{2}+\phi\nabla h(h+\delta)^\frac{(p-2)}{2}+\frac{(p-2)}{2}\phi h(h+\delta)^\frac{p-4}{2}\nabla h}^2
            \\
            \leq& C_S||\lambda^{-}||_{\nu/2} \Biggr[ \Bigl(1+\frac{1}{\epsilon}\Bigl)\int_M \abs{\nabla\phi}^2(h+\delta)^{p-2}h^2+(1+\epsilon) (1+\eta)\int_M\phi^2(h+\delta)^{p-2}\abs{\nabla h}^2
            \\&+(1+\epsilon)(1+\frac{1}{\eta})\frac{(p-2)^2}{4}\int_M\phi^2(h+\delta)^{p-4}h^2\abs{\nabla h}^2\Biggr]
    \\\leq &C_S||\lambda^{-}||_{\nu/2}\Bigl(1+\frac{1}{\epsilon}\Bigl)\int_M \abs{\nabla\phi}^2(h+\delta)^{p-2}h^2+C_S||\lambda^{-}||_{\nu/2}\bigl(1+\epsilon\bigl)\times
\\&\bigl[(1+\eta)+(1+\frac{1}{\eta})\frac{(p-2)^2}{4}\bigl] \int_M\phi^2(h+\delta)^{p-2}\abs{\nabla h}^2.
\end{aligned}
$$
    Substituting the above computations into \eqref{4.3}, we obtain
    \begin{align*}
      &\Biggr\{ p-\epsilon -C_S (1+\epsilon)\bigl[(1+\eta)+(1+\frac{1}{\eta})\frac{(p-2)^2}{4}\bigl]||\lambda^{-}||_{\nu/2} \Biggr\}\int_M \phi^2 (h+\delta)^{p-2}\abs{\nabla h}^2
    \\ &\leq \Biggr[ C_S\Bigl(1+\frac{1}{\epsilon}\Bigl)||\lambda^{-}||_{\nu/2} +\frac{1}{\epsilon} \Biggr]\int_M \abs{\nabla\phi}^2(h+\delta)^{p-2}h^2
    \\&\leq\Biggr[ C_S\Bigl(1+\frac{1}{\epsilon}\Bigl)||\lambda^{-}||_{\nu/2} +\frac{1}{\epsilon} \Biggr]\int_M\abs{\nabla\phi}^2h^p.
    \end{align*}
   Now choose $\eta=\frac{2-p}{2}$
we obtain
\begin{align*}
 &\Biggr[ p-\epsilon -(1+\epsilon)C_S||\lambda^{-}||_{\nu/2} \frac{(4-p)^2}{4}\Biggr]\int_M \phi^2 (h+\delta)^{p-2}\abs{\nabla h}^2 \\&\leq\Biggr[ C_S\Bigl(1+\frac{1}{\epsilon}\Bigl)||\lambda^{-}||_{\nu/2} +\frac{1}{\epsilon} \Biggr]\int_M\abs{\nabla\phi}^2h^p
 \\&\leq\Biggr[ C_S\Bigl(1+\frac{1}{\epsilon}\Bigl)||\lambda^{-}||_{\nu/2} +\frac{1}{\epsilon} \Biggr]\frac{d^2}{r^2}\int_Mh^p.
\end{align*}
Due to the fact that $\norm{\lambda^-}_{\nu/2} < \frac{4p}{(4-p)^2C_S}$, we can choose $\epsilon$ be a small enough potitive constant such that the coefficient of the LHS is potitive. Then by letting $r\to+\infty$, it holds true that $h=0$ on $M$ thanks to the finiteness of the $L^p$- energy of $h$ and that $M$ has infinite volume. Therefore $u$ is a contant function on $M$.
\end{proof}
\begin{remark}
    Observe that if all of the eigenvalues of the Ricci curvature tensor of $M$ are nonnegative, then the term involved with $\lambda^{-}$ in equation \eqref{4.1} shall be vanished. Hence, the assumption on the upper bound of $\norm{\lambda^{-}}_{\nu/2}$ is automatically satisfied. Moreover, it is well-known that if Ricci curvature is nonnegatuuve then a Sobolev type inequality holds true on $M$. This implies that Theorem \ref{1.5} extends a classical result by Yau (see \cite{[Yau]}) on vanishing properties of $L^p(p>1)$-harmonic functions on Riemannian manifolds with nonnegative Ricci curvature.
\end{remark}

Now we switch our attention to a vanishing theorem on manifolds with lower bound conditions on the K\"{a}hler curvature operator. To begin with, we shall follow the notation in \cite{[PW22]} and \cite{[DungCho23]} to give an explicit definition of the algebraic curvature operator on Euclidean vector spaces.

\begin{definition}
    Let $(V,g)$ be an $n$-dimensional Euclidean vector space and let $V_{\C}=V\bigotimes_{\R}\C$. For any $\R$-multilinear, complex valued tensor $T$ on $V$ and $L\in \mathfrak{so}(V)$, set
    $$LT(X_1,\ldots,X_n)= -\sum_{i=1}^r T(X_1,\ldots,LX_i,\ldots,X_r).$$
    If $\mathfrak{g}\in\mathfrak{so}(V)$ is a Lie algebra, we define $T^{\mathfrak{g}}\in\bigl( \bigotimes^r V_{\C}^{*} \bigl)\bigotimes_{\R}\mathfrak{g}$ as
    $$g(L,T^{\mathfrak{g}}(X_1,\ldots,X_r))=LT(X_1,\ldots,X_r)$$
    for every $L\in \mathfrak{g}\in\mathfrak{so}(V)=\Lambda^2(V)$.

    A tensor $Rm\in\bigotimes^4V^{*}$ is called an algebraic curvature tensor if
    $$Rm(X,Y,Z,W)=-Rm(Y,X,Z,W)=-Rm(X,Y,W,Z)=Rm(Z,W,X,Y)$$
    and
    $$Rm(X,Y,Z,W)+Rm(Y,Z,X,W)+Rm(Z,X,Y,W)=0.$$
    In particular, the curvature operator $\mathfrak{R}:\Lambda^2V\to\Lambda^2V$ is induced by
    $$g(\mathfrak{R}(\wedgeproduct{X}{Y}),\wedgeproduct{Z}{W})=Rm(X,Y,Z,W).$$
\end{definition}

In the case of Riemannian manifold, the curvature operator $\mathfrak{g}$ vanishes on the complement of the holonomy algebra $\mathfrak{hol}$. To be specific, it induces $\mathfrak{R}_{\mathfrak{hol}}:\mathfrak{hol}\to \mathfrak{hol}$ and the associated curvature tensor $R\in Sym_B^2(\mathfrak{hol})$. When $\mathfrak{hol}=\mathfrak{u}(n)$, the manifold $(M,g)$ is K\"{a}hler and the operator $\mathfrak{R}_{\mathfrak{u}(n)}:\mathfrak{u}(n)\to \mathfrak{u}(n)$ is called the K\"{a}hler curvature operator and the corresponding $R\in Sym_B^2(\mathfrak{u}(n))$ is called the K\"{a}hler curvature tensor. It is worth to note that every K\"{a}hler curvature tensor on a K\"{a}hler manifold $(M,g,J)$, where $J$ is the complex structure, satisfies
$$Rm(X,Y,Z,W)=Rm(JX,JY,Z,W)=Rm(X,Y,JZ,JW).$$

Next, we carry on by recalling the following lemma, which originally refers to the remarkable paper written by Petersen and Wink \cite{[PW22]} (for a detailed discussion, see also  \cite{[DungCho23]}).

\begin{lemma}[\cite{[DungCho23]}] \label{4.2}
    Let $(M,g)$ be a complete non-compact K\"{a}hler manifold of complex dimension $n$, let $\kappa$ be a function on $M$ such that $\kappa\leq0$. Denote by $\mu_1\leq\cdots\leq\mu_{n^2}$ the eigenvalues of the K\"{a}hler curvature operator of $(M,g)$. Suppose that at every point $p\in M$,
    $$\frac{\mu_1(p)+\cdots+\mu_n(p)}{n+1} \geq \kappa(p).$$
    Then for every $(1,0)$-harmonic form $\omega$ on $M$, we have
    $$\frac{1}{2}\Delta\abs{\omega}^2 \geq \abs{\nabla \omega}^2 +\kappa(n+1)\abs{\omega}^2.$$
\end{lemma}

The following proof of Theorem \ref{1.7} with the generalized curvature condition is a straight application of Lemma \ref{4.2}.

\begin{proof}[Proof of Theorem \ref{1.7}]
Let $\omega=du$, since $u$ is pluriharmonic, so $\omega$ is harmonic. Note that every harmonic $1$-form $\omega$ can be decomposed into
$$\omega = \omega_1 + \omega_2$$
where $\omega_1$ is a $(1,0)$-harmonic form and $\omega_2$ is a $(0,1)$-harmonic form, as a consequence of Lemma \ref{4.2}, $\omega$ satisfies the following variation of the Bochner formula
    $$\frac{1}{2}\Delta\abs{\omega}^2 \geq \abs{\nabla \omega}^2 +\kappa(n+1)\abs{\omega}^2.$$
     Since $u$ is pluriharmonic, $\omega$ satisfies the following Kato-type inequality
    $$\abs{\nabla \omega}^2 \geq 2\abs{\nabla\abs{\omega}}^2,$$
    where we used $|\omega|=|du|=|\nabla u|$ and $|\nabla \omega|^2=|\nabla^2u|$.
    Hence
    \begin{equation}
        \abs{\omega}\Delta\abs{\omega} \geq \abs{\nabla\abs{\omega}}^2 +\kappa(n+1)\abs{\omega}^2.\label{4.4}
    \end{equation}
    Let $\phi$ be the same cut-off function defined in the proof of Theorem \ref{1.5}. In the following for simplicity we let $h:=\abs{\omega}$.

    \textbf{Case $p\geq 2$.} Multiplying both side of \eqref{4.4} with $\phi^2h^{p-2}$ gives us
    $$\int_M \phi^2 h^{p-1} \Delta h \geq \int_M \phi^2 h^{p-2}\abs{\nabla h}^2 + (n+1)\int_M \kappa \phi^2h^p.$$
    By utilizing the integration by parts formula and reshuffling the terms, we have that
    \begin{equation}
        2\int_M \phi h^{p-1}\innerproduct{\nabla\phi}{\nabla h} + p\int_M \phi^2 h^{p-2} \abs{\nabla h}^2 \leq -(n+1)\int_M \kappa\phi^2 h^p.\label{4.5}
    \end{equation}
    Applying the Cauchy inequality to the first term of the LHS of \eqref{4.5}, we obtain
    $$-2\int_M \phi h^{p-1}\innerproduct{\nabla\phi}{\nabla h} \leq \epsilon\int_M \phi^2 h^{p-2} \abs{\nabla h}^2 + \frac{1}{\epsilon} \int_M h^p \abs{\nabla\phi}^2,$$
    where $\epsilon$ is a positive constant.
    Next, by using the H\"{o}lder inequality, the weighted Sobolev inequality and Young's inequality consecutively on the RHS terms of \eqref{4.5}, the following evaluation holds
    $$-\int_M \kappa \phi^2 h^p \leq C_S\norm{\kappa}_{\nu/2} \Biggr[ \frac{p^2(1+\epsilon)}{4}\int_M \phi^2 h^{p-2}\abs{\nabla h}^2 + \Bigl(1+\frac{1}{\epsilon}\Bigl)\int_M h^p \abs{\nabla\phi}^2 \Biggr].$$
    Substituting the above computations into \eqref{4.5}, we obtain
    $$\begin{aligned}\Biggr[p-\epsilon -\frac{C_S(n+1)p^2(1+\epsilon)}{4}\norm{\kappa}_{\nu/2} \Biggr] \int_M \phi^2 h^{p-2} \abs{\nabla h}^2 
    \\\leq \Biggr[ C_S(n+1)\Bigl(1+\frac{1}{\epsilon}\Bigl)\norm{\kappa}_{\nu/2} +\frac{1}{\epsilon} \Biggr]\int_M h^p \abs{\nabla\phi}^2.
    \end{aligned}$$
    Since $\norm{\kappa}_{\nu/2} < \frac{4}{pC_S(n+1)}$, it it true that for sufficiently small  positive constant $\epsilon$ we have
    $$D(n,\epsilon) = \frac{p -\epsilon -\frac{C_S(n+1)p^2(1+\epsilon)}{4}\norm{\kappa}_{\nu/2}}{C_S(n+1)\Bigl(1+\frac{1}{\epsilon}\Bigl)\norm{\kappa}_{\nu/2} +\frac{1}{\epsilon}} >0.$$
    Hence ,
    $$D(n,\epsilon)\int_{B_o(r)} \phi^2 \abs{\omega}^{p-2} \abs{\nabla \abs{\omega}}^2 \leq \frac{d^2}{r^2}\int_{B_o(r)} \abs{\omega}^p.$$
    By letting $r\to+\infty$, the result follows from the fact that $\abs{\omega}\in L^p$ and that $M$ has infinite volume.

    \textbf{Case $1<p<2.$} Let
 $\delta,\eta$ be positive constants in the following.  Multiplying both side of \eqref{4.4} with $\phi^2(h+\delta)^{p-2}$ gives us
    $$\int_M \phi^2 (h+\delta)^{p-2} h\Delta h \geq \int_M \phi^2 (h+\delta)^{p-2}\abs{\nabla h}^2 + (n+1)\int_M \kappa \phi^2(h+\delta)^{p-2}h^2.$$
   By utilizing the integration by parts formula and reshuffling the terms, we have that
   \begin{equation}
    \begin{aligned}
         &2\int_M \phi (h+\delta)^{p-2} h\innerproduct{\nabla\phi}{\nabla h} + (p-2)\int_M \phi^2 (h+\delta)^{p-3}h \abs{\nabla h}^2\\
         &+2\int_M\phi^2 (h+\delta)^{p-2}\abs{\nabla h}^2
         \leq-(n+1)\int_M \kappa\phi^2(h+\delta)^{p-2}h^2.\label{4.6}
    \end{aligned}
    \end{equation}
    Note that
  $$\int_M\phi^2 (h+\delta)^{p-2}\abs{\nabla h}^2\geq\int_M \phi^2 (h+\delta)^{p-3}h \abs{\nabla h}^2,$$
and applying the Cauchy inequality  gives us
    $$-2\int_M \phi (h+\delta)^{p-2}h\innerproduct{\nabla\phi}{\nabla h} \leq \epsilon\int_M \phi^2 (h+\delta)^{p-2} \abs{\nabla h}^2 + \frac{1}{\epsilon} \int_M (h+\delta)^{p-2}h^2 \abs{\nabla\phi}^2,$$
    where $\epsilon$ is a positive constant. Next, by using H\"{o}lder inequality, the weighted Sobolev inequality and Young's inequality consecutively, the following evaluation holds
    $$
        \begin{aligned}
            &\int_M -k\phi^2 (h+\delta)^{p-2}h^2 \leq \Biggl(\int_M (-k)^\frac{\nu}{2}\Biggl)^\frac{2}{\nu} \Biggl(\int_M \Bigl(\phi h(h+\delta)^{\frac{p-2}{2}}\Bigl)^{\frac{2\nu}{\nu-2}}\Biggl)^{\frac{\nu-2}{\nu}} \\
             \leq& C_S||k||_{\nu/2} \int_M \Big| \nabla \Bigl(\phi h(h+\delta)^{\frac{p-2}{2}}\Bigl) \Big|^2
        \\=&C_S||k||_{\nu/2}
            \int_M\abs{\nabla\phi h(h+\delta)^\frac{p-2}{2}+\phi\nabla h(h+\delta)^\frac{(p-2)}{2}+\frac{(p-2)}{2}\phi h(h+\delta)^\frac{p-4}{2}\nabla h}^2
            \\
             \leq& C_S||k||_{\nu/2} \Biggr[ \Bigl(1+\frac{1}{\epsilon}\Bigl)\int_M \abs{\nabla\phi}^2(h+\delta)^{p-2}h^2+(1+\epsilon) (1+\eta)\int_M\phi^2(h+\delta)^{p-2}\abs{\nabla h}^2
            \\&+(1+\epsilon)(1+\frac{1}{\eta})\frac{(p-2)^2}{4}\int_M\phi^2(h+\delta)^{p-4}h^2\abs{\nabla h}^2\Biggr]
    \\\leq &C_S||k||_{\nu/2}\Bigl(1+\frac{1}{\epsilon}\Bigl)\int_M \abs{\nabla\phi}^2(h+\delta)^{p-2}h^2+C_S||k||_{\nu/2}\bigl(1+\epsilon\bigl)\times
\\&\left[(1+\eta)+(1+\frac{1}{\eta})\frac{(p-2)^2}{4}\right] \int_M\phi^2(h+\delta)^{p-2}\abs{\nabla h}^2.
\end{aligned}
$$
 Substituting the above computations into \eqref{4.6}, we obtain
    \begin{align*}
      &\Biggr\{ p-\epsilon -(n+1)C_S (1+\epsilon)\bigl[(1+\eta)+(1+\frac{1}{\eta})\frac{(p-2)^2}{4}\bigl]||k||_{\nu/2} \Biggr\}\int_M \phi^2 (h+\delta)^{p-2}\abs{\nabla h}^2
    \\ &\leq \Biggr[ (n+1)C_S\Bigl(1+\frac{1}{\epsilon}\Bigl)||k||_{\nu/2} +\frac{1}{\epsilon} \Biggr]\int_M \abs{\nabla\phi}^2(h+\delta)^{p-2}h^2
    \\&\leq\Biggr[ (n+1)C_S\Bigl(1+\frac{1}{\epsilon}\Bigl)||k||_{\nu/2} +\frac{1}{\epsilon} \Biggr]\int_M\abs{\nabla\phi}^2h^p.
    \end{align*}
   Now choose $\eta=\frac{2-p}{2}$
we obtain \begin{align*}
 &\Biggr[ p-\epsilon -(n+1)C_S(1+\epsilon)||k||_{\nu/2} \frac{(4-p)^2}{4}\Biggr]\int_M \phi^2 (h+\delta)^{p-2}\abs{\nabla h}^2 \\&\leq\Biggr[ (n+1)C_S\Bigl(1+\frac{1}{\epsilon}\Bigl)||k||_{\nu/2} +\frac{1}{\epsilon} \Biggr]\int_M\abs{\nabla\phi}^2h^p
 \\&\leq\Biggr[ (n+1)C_S\Bigl(1+\frac{1}{\epsilon}\Bigl)||k||_{\nu/2} +\frac{1}{\epsilon} \Biggr]\frac{d^2}{r^2}\int_Mh^p.
\end{align*}
Due to the fact that $\norm{k}_{\nu/2} < \frac{4p}{(4-p)^2C_S(n+1)}$, we can choose $\epsilon$ be a sufficiently small potitive constant such that the coefficient of the LHS is potitive.  Then by letting $r\to+\infty$, the result follows from the fact that $\abs{\omega}\in L^p$ and that $M$ has infinite volume.
\end{proof}
\begin{remark}
    We note that in \cite{[DungCho23]}, the second author and Cho obtained a vanishing result for general $(p, q)$ forms $\omega$ with finite $L^Q \ (Q\geq    2)$ energy. However, due to the defection of a refined Kato inequality for usual harmonic forms, they required the forms to be harmonic fields. Note that this assumption is stronger than the usual harmonicity. In our theorem, we used the usual harmonicity of the function $u$ and do not need to assume that $\omega=du$ is a harmonic field.
\end{remark}

    %%%%%%%%%%%%%5
    %%%%%%%%%%%%%
\section*{Acknowledgement}
A part of this work was completed during a stay of the second author at Vietnam Institute for Advanced Study in Mathemmatics (VIASM) in summer 2023. He would like to thank the staff there for hospitality and support. The second author was supported by NAFOSTED under grant number 101.02-2021.28. The third author was supported by the NSF of China (No.12271069).

\end{document}